\documentclass[12pt,reqno]{amsart}

\newcommand\version{September 26, 2025}

\usepackage{amsmath,amsfonts,amsthm,amssymb,amsxtra}
\usepackage{bbm} 
\usepackage{hyperref} 

\newtheorem{theorem}{Theorem}
\newtheorem{proposition}[theorem]{Proposition}
\newtheorem{lemma}[theorem]{Lemma}

\theoremstyle{definition}

\theoremstyle{remark}

\newtheorem{remark}[theorem]{Remark}

\renewcommand{\epsilon}{\varepsilon}

\newcommand{\loc}{{\rm loc}}

\newcommand{\N}{\mathbb{N}}

\renewcommand{\phi}{\varphi}
\newcommand{\R}{\mathbb{R}}

\begin{document}

\title[Openness of shape optimizers]{Openness of shape optimizers\\ in higher-order and non-scalar problems}

\author{Rupert L. Frank}
\address[Rupert L. Frank]{Mathe\-matisches Institut, Ludwig-Maximilians Universit\"at M\"unchen, The\-resienstr.~39, 80333 M\"unchen, Germany, and Munich Center for Quantum Science and Technology, Schel\-ling\-str.~4, 80799 M\"unchen, Germany, and Mathematics 253-37, Caltech, Pasa\-de\-na, CA 91125, USA}
\email{r.frank@lmu.de}

\thanks{\version\\
	\copyright\, 2025 by the author. This paper may be reproduced, in its entirety, for non-commercial purposes.}

\begin{abstract}
	We consider the first eigenvalues of the polyharmonic, Lam\'e and Stokes operators with Dirichlet boundary conditions on sets of given finite measure. It is shown that a quasi-open set for which this eigenvalue is minimal is open. This removes dimensional restrictions in earlier works. We use Campanato theory, which works well in the present higher order or non-scalar setting.
\end{abstract}

\maketitle

\section{Introduction and main result}

Often in the calculus of variation, when trying to prove the existence of a minimizer for some functional, one relaxes the class of admissible configurations, so that the new class has better compactness properties and a minimizer in this enlarged class is easier to find. The prize to be paid for this convenience is an extra step in the proof where one needs to show that the minimizer, which initially only belongs to the new, enlarged class, in fact lies in the original class of configurations. This is often related to regularity considerations.

Specifically, we are interested in shape optimization problems in the class of open subsets of Euclidean space. When the set-functional to be minimized involves Sobolev spaces, a frequently used relaxation is to quasi-open sets. These are sets that can be approximated from the outside by open sets in the sense of a certain capacity (depending on the Sobolev space under consideration). Using tools from concentration compactness, it is often possible to show the existence of a minimizing set in the class of quasi-open sets; see, e.g., \cite{Bu} for a survey. In this paper we are concerned with the remaining regularity step, which consists in showing that the quasi-open set is open.

We are motivated by two recent papers, \cite{Le,HeLePr}, where openness of quasi-open minimizers was shown only in small dimensions. Our goal is to remove this dimensional restriction. The difficulty is that these problems involve Sobolev spaces of order larger than one and Sobolev spaces of vector-valued functions, respectively. In both cases tools based on the maximum principle are not available, thus precluding an approach to showing openness that works well in the case of a Sobolev space of first order for scalar functions.

We shall treat simultaneously the shape optimization problems from \cite{Le,HeLePr}, as well as \cite{HeMaPr}. They are referred to as Problems I, II and III and concern the lowest eigenvalue of the polyharmonic, the Lam\'e and the Stokes operator, respectively. For an open set $\Omega\subset\R^d$, $d\geq 2$, we consider the energy space
$$
V(\Omega) :=
\begin{cases}
	H^m_0(\Omega,\R) & \text{in Problem I} \,,\\
	H^1_0(\Omega,\R^d) & \text{in Problem II} \,,\\
	\{ u\in H^1_0(\Omega,\R^d):\ \nabla\cdot u = 0 \} & \text{in Problem III} \,.
\end{cases}
$$
Here $m\in\N$ is an integer (the case $m\geq 2$ being the new one) and $\nabla\cdot$ denotes the divergence. In what follows it will be convenient to set $m=1$ in Problems II and III. Associated to each element $u\in V(\Omega)$ is an energy density
$$
\tau(u) :=
\begin{cases}
	 \sum_{i\in\{1,\ldots,d\}^m} |\partial_{i_1}\cdots\partial_{i_m} u|^2 & \text{in Problem I} \,,\\
	|\nabla u|^2 + c_0(\nabla\cdot u)^2 & \text{in Problem II} \,,\\
	|\nabla u|^2 & \text{in Problem III} \,.
\end{cases}
$$
Here $c_0\in\R_+=(0,\infty)$ is a fixed parameter. We are interested in
\begin{equation}
	\label{eq:deflambdaomega}
	\Lambda(\Omega) := \inf_{0\neq u\in V(\Omega)} \frac{\int_\Omega \tau(u) \,dx}{\int_\Omega |u|^2\,dx} \,.
\end{equation}
When attained, $\Lambda(\Omega)$ gives the smallest eigenvalue on $\Omega$ of the polyharmonic operator in Problem I, the Lam\'e operator in Problem II and the Stokes operator in Problem~III, each considered with Dirichlet boundary conditions.

The shape optimization problem under consideration consists in \emph{minimizing $\Lambda(\Omega)$ among all open sets $\Omega\subset\R^d$ of given measure}. In view of the scaling behavior
\begin{equation}
	\label{eq:scaling}
	\Lambda(t\Omega) = t^{-2m} \Lambda(\Omega) 
	\qquad\text{for all}\ t>0
\end{equation}
the specific value of the measure that is fixed is irrelevant.

The definition of the space $V(\Omega)$, and consequently the definition of $\Lambda(\Omega)$, can be extended to quasi-open sets $\Omega$. The notion `quasi-open' depends on $m$ and its precise definition is not relevant here. It easily follows from Sobolev embedding theorems that $|\Omega|^{2m/d}\Lambda(\Omega)$ is bounded from below by a positive constant that depends only on $d$ and~$m$.

It is shown in \cite{Le} for Problem~I, in \cite{HeMaPr} for Problem II and in \cite{HeLePr} for Problem III that there is a set minimizing $\Lambda$ among all quasi-open sets of given measure. (In \cite{HeMaPr} only the case of dimensions $d=2,3$ is considered, but their argument generalizes to arbitrary $d$.) For any quasi-open set $\Omega$ of finite measure and, in particular, for any quasi-open volume constrained minimizer the embedding $V(\Omega)\subset L^2(\Omega)$ is compact and therefore $\Lambda(\Omega)$ is an eigenvalue and the infimum in \eqref{eq:deflambdaomega} is attained by a corresponding eigenfunction. Studying the regularity of this eigenfunction will be what gives us regularity of the underlying set.

The following is our main result.

\begin{theorem}\label{main}
	Let $d\geq 2$ and let $\Omega\subset\R^d$ be a quasi-open volume constrained minimizer for $\Lambda$ in one of Problems I, II or III. Then, up to a set of measure zero, $\Omega$ is open and any eigenfunction $u$ of $\Lambda(\Omega)$, extended by zero to $\R^d$, belongs to $C^{m-1,\alpha}(\R^d)$ for any $\alpha<1$ and we have $D^k u\in L^\infty(\R^d)$ for any $k\in\{0,\ldots,m-1\}$.
\end{theorem}

This result is proved in \cite{Le} for Problem I under the additional restriction $d \leq 4m$ and in \cite{HeLePr} for Problem II under the additional restriction $d\leq 3$. (Our result also improves the range of H\"older exponents in Problem II for $d=3$ from \cite{HeLePr}.)
Thus, our result proves openness \emph{without any dimensional restriction}.

We emphasize that the result is nontrivial: not every quasi-open set is equal, up to a set of measure zero, to an open set; see \cite[Exercise 3.6]{HePi} for an example.

We do not expect that eigenfunctions $u$ for $\Lambda(\Omega)$ with an optimal set $\Omega$ belong to $C^m(\R^d)$ (because we do not expect $D^m u$ to vanish on $\partial\Omega$). It is natural to wonder whether they belong to $C^{m-1,1}(\R^d)$. This is known in variations of this problem in the scalar case for $m=1$ (see \cite{BrHaPi} and also \cite{BuMaPrVe,LaPi}), but we do not see how to extend these techniques to the higher-order or non-scalar cases. The best we can get with our techniques is that $D^m u$ belongs to a BMO-type space; see Remark \ref{rem:bmo}.

The method that we use to prove Theorem \ref{main} and that we describe momentarily is not restricted to our three problems. It works for mixed order elliptic systems that are translation invariant and satisfy the Legendre--Hadamard ellipticity condition, as considered for instance in \cite{GiMo}. This constitutes a common generalization of Problems~I and II. Since existence of quasi-open volume constrained minimizers has not been proved in this setting, we refrain from stating our regularity results in this generality.


\subsection*{Strategy of the proof}

Given $J:V(\R^d)\to \R$ and constants $C,c,R>0$, we say that $u\in V(\R^d)$ is a \emph{local quasiminimizer} of $J$ with parameters $(C,c,R)$ if for any $a\in\R^d$, $r\in(0,R]$ and $v\in V(\R^d)$ with $v-u\in V(B_r(a))$ and $\int_{B_r(a)} \tau(v-u)\,dx \leq c$ one has
$$
J(v) \geq J(u) - C r^d \,.
$$
We simply say that $u$ is a local quasiminimizer if the values of the constants $C,c,R$ are irrelevant. This notion appears for example in \cite{BuMaPrVe,Le}.

For $f\in V'(\R^d)$, the dual space of $V(\R^d)$, we introduce the functional
$$
J_f(v):= \int_{\R^d}\tau(u)\,dx - 2 \langle f,u \rangle \,.
$$

\begin{proposition}\label{quasimin}
	Let $\Omega\subset\R^d$ be a quasi-open volume constrained minimizer for $\Lambda$ in one of Problems I, II or III and let $u$ be an $L^2$-normalized eigenfunction of $\Lambda(\Omega)$, extended by zero to $\R^d$. Then $u$ is a local quasiminimizer of $J_f$ with $f=\Lambda(\Omega)\,u$.
\end{proposition}

The proof of this proposition is rather standard; see for instance \cite{BuMaPrVe}. For Problem~I it appears in \cite[Proof of Theorem~5]{Le}. For the sake of completeness we provide the details in Section \ref{sec:quasimin}.

We now discuss the regularity of local quasiminimizers of functions $J_f$ in dependence on the regularity of $f$. This is most naturally carried out in terms of the scale of Morrey spaces. The Morrey space $L^{2,\lambda}(\R^d)$ with parameter $0\leq \lambda\leq d$ consists of all $f\in L^2_\loc(\R^d)$ such that
$$
\| f \|_{L^{2,\lambda}(\R^d)} := \sup_{0<r\leq r_0\,,\, a\in\R^d} \left( r^{-\lambda} \int_{B_r(a)} |f(x)|^2\,dx \right)^{1/2} <\infty \,.
$$
Here $r_0>0$ is a (finite) constant. A covering argument shows that different choices of $r_0$ lead to equivalent norms. The same definition applies when $f$ is vector- or matrix-valued.

\begin{proposition}\label{regularity}
	Let $0\leq \lambda\leq d$ and $f\in L^{2,\lambda}(\R^d)$. If $u$ is a local quasiminimizer for $J_f$, then 
	$$
	D^m u\in L^{2,\mu}(\R^d)
	\qquad
	\begin{cases}
		\text{with}\ \mu=\lambda +2m & \text{if}\ \lambda+2m<d \,,\\
		\text{with any}\ \mu<d & \text{if}\ \lambda + 2m \geq d \,.
	\end{cases}
	$$
\end{proposition}

The proof of this proposition is the content of Section \ref{sec:regularity}.

Given Propositions \ref{quasimin} and \ref{regularity}, our main result, Theorem \ref{main}, follows by a bootstrap argument in the scale of Morrey spaces. We provide the details of this argument in Section \ref{sec:proofmain}.

To prove Proposition \ref{regularity} we use an approach to regularity theory due to Campanato, which is well-known to work in higher order and non-scalar settings. We refer to \cite[Chapter 5]{GiMa} for an introduction to this technique.

Campanato theory is also used in the papers \cite{Le} and \cite{HeLePr} that motivated our work and, indeed, from a wider perspective our techniques are similar to those employed there. There are, however, some important differences in the details that allow us to obtain sharper results and which, we hope, will be useful in other, related shape optimization problems. The notion of local quasiminimality that we adopt is used in \cite{Le}, where also a version of Proposition \ref{regularity} for $f$ in an $L^p$ space is proved. Thus, the novelty of our approach compared to \cite{Le} is the use of Morrey spaces, as well as the idea of repeatedly applying Proposition \ref{regularity} to improve the regularity. In \cite{HeLePr} a different notion of quasiminimality is used, which does not involve an inhomogenity $f$ and therefore seems not to be amenable to a bootstrap procedure.

Finally, we mention the work \cite{St} concerning the volume constrained minimization of the buckling eigenvalue, where also openness of optimal sets and the $C^{m-1,\alpha}$ regularity of corresponding eigenfunctions is proved. The proof also proceeds by Campanato theory and bootstrap, but does not use the notion of local quasiminimality. It proves existence and regularity simultaneously via a penalization procedure and therefore follows a somewhat different strategy than \cite{Le,HeLePr}.


\section{Local quasiminimizers}\label{sec:quasimin}

In this section we prove Proposition \ref{quasimin}.

\begin{proof}
	We shall show that for any $M>0$ and any $\beta>0$ there is a constant $C_{M,\beta}$ with the following property: if $\Omega\subset\R^d$ is a quasi-open volume constrained minimizer for $\Lambda$ and if $u$ is an $L^2$-normalized eigenfunction of $\Lambda(\Omega)$, extended by zero to $\R^d$, then $u$ is a local quasiminimizer of $J_f$ with $f:=\Lambda(\Omega) u$ with parameters $(C_{M,\beta}|\Omega|^{-1-2m/d},\beta |\Omega|^{-2m/d}, M|\Omega|^{1/d})$.
	
	Let $a\in\R^d$, $r\in\R_+$ and $v\in V(\R^d)$ with $v-u\in V(B_r(a))$. The definition of the energy space for quasi-open sets implies that $v\in V(\Omega\cup B_r(a))$. Then, by homogeneity \eqref{eq:scaling} and by minimality of $\Omega$,
	\begin{align*}
		\frac{\int_{\R^d} \tau(v)\,dx}{\int_{\R^d} |v|^2\,dx} & \geq \Lambda(\Omega\cup B_r(a))
		= \left( \frac{|\Omega|}{|\Omega\cup B_r(a)|}\right)^{2m/d}  \,\Lambda \left( \frac{|\Omega|^{1/d}}{|\Omega\cup B_r(a)|^{1/d}}\, (\Omega\cup B_r(a)) \right) \\
		& \geq \left( \frac{|\Omega|}{|\Omega\cup B_r(a)|}\right)^{2m/d} \Lambda(\Omega) \,. 
	\end{align*}
	We have
	$$
	\left( \frac{|\Omega|}{|\Omega\cup B_r(a)|}\right)^{2m/d} = \left( 1 + \frac{| B_r(a)\setminus\Omega|}{|\Omega|} \right)^{-2m/d} \geq 1 - \frac{2m}d \frac{| B_r(a)\setminus\Omega|}{|\Omega|} \geq 1 - \frac{2m}d \frac{|B_1| r^d}{|\Omega|} \,.
	$$
	Thus,
	$$
	\int_{\R^d} (\tau(v) - \Lambda(\Omega) |v|^2)\,dx \geq - \frac{2m}d \frac{|B_1| r^d}{|\Omega|} \Lambda(\Omega) \int_{\R^d} |v|^2\,dx \,.
	$$
	Using $J_f(u) = -\Lambda(\Omega) \int_{\R^d}|u|^2\,dx$, we can bound the left side from above by
	$$
	\int_{\R^d} (\tau(v) - \Lambda(\Omega) |v|^2)\,dx = J_f(v) - J_f(u)- \Lambda(\Omega) \int_{\R^d} |v-u|^2\,dx \leq J_f(v) - J_f(u) \,,
	$$
	so we obtain
	\begin{equation*}
		J_f(v) - J_f(u) \geq - \frac{2m}d \frac{|B_1| r^d}{|\Omega|} \Lambda(\Omega) \int_{\R^d} |v|^2\,dx \,.
	\end{equation*}
	It remains to bound the $L^2$-norm on the right side. Using the normalization of $u$, we have
	$$
	\left( \int_{\R^d} |v|^2\,dx \right)^{1/2} \leq 1 +  \left( \int_{\R^d} |v-u|^2\,dx \right)^{1/2} \leq 1 + \Lambda(B_r)^{-1/2} \left( \int_{\R^d} \tau(v-u)\,dx \right)^{1/2}.
	$$
	If $r\leq M|\Omega|^{1/d}$ and $\int_{\R^d} \tau(v-u)\,dx \leq \beta |\Omega|^{-2m/d}$, then the right side is bounded by $1 + \Lambda(B_1)^{-1/2} M^m \beta^{1/2}$. This completes the proof with
	$$
	C_{M,\beta} =  \frac{2m}d |B_1| |\Omega|^{2m/d} \Lambda(\Omega) (1+ \Lambda(B_1)^{-1/2} M^m \beta^{1/2})^2 \,.
	$$
	Note that $|\Omega|^{2m/d} \Lambda(\Omega)$ is a constant depending only on $d$, $m$ and $c_0$.
\end{proof}


\section{Regularity}\label{sec:regularity}

As a preparation for the proof of Proposition \ref{regularity} we record the following basic property of local quasiminimizers. For Problem I it appears in \cite[Lemma 18]{Le}; see also \cite{BuMaPrVe}. We let $\tau(\cdot,\cdot)$ denote the bilinear form associated to the quadratic form $\tau(\cdot)$.

\begin{lemma}\label{quasiminbasic}
	Let $f\in V'(\R^d)$ and let $u\in V(\R^d)$ be a local quasiminimizer for $J_f$ with parameters $(C,c,R)$. Then, for all $a\in\R^d$, $r\in(0,R]$ and $\phi\in V(B_r(a))$,
	$$
	\left| \int_{\R^d} \tau(u,\phi) \,dx - \langle f,\phi \rangle \right| \leq C' r^{d/2} \left( \int_{\R^d} \tau(\phi) \,dx \right)^{1/2}
	$$
	with $C'$ depending only on $C$, $c$ and $R$.
\end{lemma}

\begin{proof}
	We choose $v=u+\gamma\phi$ in the definition of local quasiminimality. The constant $\gamma$ is chosen such that
	$$
	\gamma^2  \int_{B_r(a)} \tau(\phi)\,dx \leq\alpha \,,
	$$
	which guarantees the required bound on $v-u$. We deduce that
	$$
	\gamma^2 \int_{\R^d} \tau(\phi)\,dx +2\gamma \int_{\R^d} \tau(u,\phi)\,dx - 2 \gamma \langle f,\phi \rangle  = J(u+\gamma\phi) - J(u)  \geq - C r^d \,.
	$$
	Optimizing over the sign of $\gamma$, this leads to
	$$
	\left| \int_{\R^d} \tau(u,\phi)\,dx - \langle f,\phi \rangle \right| \leq \frac12 |\gamma|^{-1} C r^d + \frac12 |\gamma|  \int_{\R^d} \tau(\phi)\,dx \,.
	$$
	If $Cr^d \leq \alpha$, we choose $\gamma^2 = Cr^d (\int_{\R^d} \tau(\phi)\,dx )^{-1}$ and obtain the claimed bound with $C'=1$. If $Cr^d>\alpha$, we choose $\gamma^2 = \alpha (\int_{\R^d} \tau(\phi)\,dx )^{-1}$ and obtain the claimed bound with constant $\frac12 \alpha^{-1/2} C r^{d/2} + \frac12 \alpha^{1/2} r^{-d/2} \leq \alpha^{-1/2} C r^{d/2} \leq \alpha^{-1/2} C R^{d/2}=:C'$.
\end{proof}

We now turn to the proof of the regularity result.

\begin{proof}[Proof of Proposition \ref{regularity}]
	Let $u$ be a local quasiminimizer of $J_f$ with parameters $(C,c,R)$. Let $a\in\R^d$ and $r\in(0,R]$. A simple compactness argument shows that the infimum
	$$
	\inf\left\{ \int_{B_r(a)} \tau(z-u)\,dx :\ z\in V(B_r(a)) \right\}
	$$
	is attained by some $w$; see for example \cite[Corollary 3.46]{GiMa} in the setting of Problem II.

	For any $\rho\in(0,r]$, we use
	\begin{equation}
		\label{eq:regproof}
		\int_{B_\rho(a)} \tau(u)\,dx \leq 2 \int_{B_\rho(a)} \tau(u-w) \,dx + 2 \int_{B_\rho(a)} \tau(w) \,dx
	\end{equation}
	and bound separately the two terms on the right side.
	
	For the first term, we use the fact that the function $v:=u-w$ satisfies
	\begin{equation}
		\label{eq:regularityequation}
		\int_{B_r(a)} \tau(z,v)\,dx = 0 
		\qquad\text{for all}\ z\in V(B_r(a)) \,.
	\end{equation}
	Indeed, this is the Euler equation of the minimization problem defining $w$. As a consequence of solving this equation, we have
	\begin{equation}
		\label{eq:decayharmonic}
			\int_{B_\rho(a)} \tau(v)\,dx \leq C'' \left( \frac{\rho}{r} \right)^d \int_{B_r(a)} \tau(v)\,dx
		\qquad\text{for all}\ \rho\in(0,r]
	\end{equation}
	with a constant $C''$ that only depends on $d$, $m$ and the parameter $c_0$. Indeed, for Problem I this follows from \cite[Eq.~(3.2)]{GiMo}, for Problem II from \cite[Proposition 5.8]{GiMa} and for Problem III from Lemma \ref{harmonicstokes} below. Moreover, by minimality of $w$ we can bound
	$$
	\int_{B_r(a)} \tau(v)\,dx = \int_{B_r(a)} \tau(u-w)\,dx \leq \int_{B_r(a)} \tau(u)\,dx \,.
	$$
	
	To bound the second term on the right side of \eqref{eq:regproof}, we extend the integral over $B_r(a)$ and use the equation satisfied by $v$. In \eqref{eq:regularityequation} we choose $z=w$ and obtain
	\begin{align*}
		\int_{B_r(a)} \tau(w)\,dx & = - \int_{B_r(a)} \tau(u,w)\,dx \\
		& = - \int_{B_r(a)} (\tau(u,w) - f\cdot w)\,dx - \int_{B_r(a)} f\cdot w\,dx \,.
	\end{align*}
	For the first term we have, according to Lemma \ref{quasiminbasic},
	$$
	- \int_{B_r(a)} (\tau(u,w) - f\cdot w)\,dx \leq C' r^{d/2} \left( \int_{B_r(a)} \tau(w)\,dx \right)^{1/2},
	$$
	while for the second term we have
	\begin{align*}
		- \int_{B_r(a)} f\cdot w\,dx & \leq \left( \int_{B_r(a)} |f|^2\,dx \right)^{1/2} \left( \int_{B_r(a)} |w|^2\,dx \right)^{1/2} \\
		& \leq \Lambda(B_1)^{-1/2} r^{m+\lambda/2} \|f\|_{L^{2,\lambda}} \left( \int_{B_r(a)} \tau(w) \,dx \right)^{1/2}.
	\end{align*}
	(We choose $r_0=R$ in the definition of the Morrey norm.) Thus, we have shown that
	\begin{equation}
		\label{eq:regularityproof2}
			\int_{B_r(a)} \tau(w)\,dx \leq 2 C' r^d + 2 \Lambda(B_1)^{-1} r^{2m+\lambda} \|f\|_{L^{2,\lambda}}^2 \,.
	\end{equation}
	
	Combining the bounds on the first and second term on the right side of \eqref{eq:regproof}, we see that $\Phi(\rho):= \int_{B_\rho(a)} \tau(u)\,dx$ satisfies
	\begin{equation}
		\label{eq:iterationineq}
		\Phi(\rho) \leq 2 C'' \left( \frac\rho r \right)^d \Phi(r) + 4 C' r^d + 4 \Lambda(B_1)^{-1} r^{2m+\lambda} \|f\|_{L^{2,\lambda}}^2 \,.
	\end{equation}
	This inequality is valid for all $0<\rho\leq r\leq R$. We want to apply the iteration lemma \cite[Lemma 5.13]{GiMa} and the crucial assumption there is that the exponent of $\rho$ in front of the $\phi$-term on the right side is strictly larger than the exponent of $r$ in the remaining terms on the right side. Thus, let $\mu=\lambda+2m$ if $\lambda+2m<d$ and fix any $\mu<d$ if $\lambda+2m\geq d$. Then the above inequality implies that, for all $0<\rho\leq r\leq R$,
	$$
	\Phi(\rho) \leq 2 C'' \left( \frac\rho r \right)^d \Phi(r) + \left( 4 C' R^{d-\mu} + 4 \Lambda(B_1)^{-1} R^{2m+\lambda-\mu} \|f\|_{L^{2,\lambda}}^2 \right) r^\mu \,.
	$$
	The iteration lemma yields a constant $C'''$, depending only on $d$, $m$, $c_0$ and $\mu$, such that, for all $0<\rho\leq r\leq R$,
	$$
	\Phi(\rho) \leq C''' \left( r^{-\mu} \Phi(r) + \left( 4 C' R^{d-\mu} + 4 \Lambda(B_1)^{-1} R^{2m+\lambda-\mu} \|f\|_{L^{2,\lambda}}^2 \right) \right) \rho^\mu \,.
	$$
	In this inequality we may choose $r=R$ and deduce that $D^m u\in L^{2,\mu}$, as claimed.
\end{proof}

It remains to prove the following fact about `Stokes-harmonic' functions.

\begin{lemma}\label{harmonicstokes}
	Let $R>0$, $v\in H^1(B_R,\R^d)$ and $p\in L^2(B_R)$ be such that
	$$
	\begin{cases}
		-\Delta v + \nabla p = 0 & \text{in}\ B_R \,,\\
		\nabla\cdot v = 0  & \text{in}\ B_R \,.
	\end{cases}
	$$
	Then
	$$
	\int_{B_r} |\nabla v|^2\,dx \lesssim \left( \frac{r}{R} \right)^d \int_{B_R} |\nabla v|^2\,dx
	\qquad\text{for all}\ 0<r\leq R \,,
	$$
	with an implied constant that depends only on $d$.
\end{lemma}

\begin{proof}
	We shall prove that
	$$
	\sup_{B_{R/2}} |\nabla v|^2 \lesssim R^{-d} \int_{B_R} |\nabla v|^2\,dx \,.
	$$
	This immediately implies the claimed inequality for $r\in(0,R/2]$, and the bound for $r\in(R/2,R]$ follows trivially.
	
	Let $\omega_{ij}:=\partial_i v_j - \partial_j v_i$. Since $\nabla\cdot v=0$, we have
	$$
	-\Delta v_i = \sum_j \partial_j \omega_{ij}
	\qquad\text{in}\ B_R \,.
	$$
	Differentiating this equation and applying a simple $L^\infty$ estimate for Poisson's equation (see \cite[Theorem 8.17]{GiTr} for a much more general result), we find
	$$
	\sup_{B_{R/2}} \, (\partial_k v_i)^2 \lesssim R^{-d} \int_{B_{3R/4}} (\partial_k v_i)^2\,dx + R^4 \sup_{B_{3R/4}} \left(  \sum_j \partial_{jk}^2 \omega_{ij} \right)^2 \,.
	$$
	
	Next, the equation $-\Delta v +\nabla p =0$ implies
	$$
	-\Delta \omega_{ij} = 0
	\qquad\text{in}\ B_R \,,
	$$
	and, consequently, by bounds for harmonic functions \cite[Theorems 2.1 and 2.10]{GiTr},
	$$
	\sup_{B_{3R/4}} |\nabla^2 \omega_{ij}|^2 \lesssim R^{-d-4} \int_{B_R} \omega_{ij}^2\,dx \,.
	$$
	The claimed bound follows by noting that $\omega_{ij}^2\lesssim |\nabla v|^2$.
\end{proof}	

In the following proposition we improve the conclusion of Proposition \ref{regularity} in case $\lambda+2m\geq d$ to a BMO-type condition. For a (scalar or non-scalar) function $f$ we write
$$
(f)_{a,r} := |B_r(a)|^{-1} \int_{B_r(a)} f(x)\,dx \,.
$$

\begin{proposition}\label{regularitybmo}
	If $\lambda +2m \geq d$ in the setting of Proposition \ref{regularity}, then
	\begin{equation}
		\label{eq:regularitybmo}
			\sup_{0<r\leq r_0,\, a\in\R^d} r^{-d} \int_{B_r(a)} |D^m u(x) - (D^m u)_{a,r}|^2\,dx <\infty
	\end{equation}
\end{proposition}

\begin{proof}
	We only explain the main changes compared to the proof of Proposition \ref{regularity}. Again, let $u$ be a local quasiminimizer for $J_f$ with parameters $(C,c,R)$ and again fix $a\in\R^d$ and $r\in(0,R]$ and define $w$ as before. Instead of \eqref{eq:regproof} we now bound, for any $\rho\in(0,r]$
	\begin{align}
		\label{eq:regproofbmo}
			\int_{B_\rho(a)} |Tu-(Tu)_{a,\rho}|^2 \,dx & \leq 2 \int_{B_\rho(a)} |T(w+u)-(T(w+u))_{a,\rho}|^2 \,dx \notag \\
			& \quad + 2 \int_{B_\rho(a)} |Tw-(Tw)_{a,\rho}|^2 \,dx \notag \\
			& \leq 2 \int_{B_\rho(a)} |T(w+u)-(T(w+u))_{a,\rho}|^2 \,dx \notag \\
			& \quad + 2 \int_{B_\rho(a)} \tau(w) \,dx \,.
	\end{align}
	Here $Tu$ is a tensor satisfying $|Tu|^2 = \tau(u)$ pointwise. More precisely, in Problem~I we set $Tu = D^m u$, that is, $Tu$ has entries $T_{i_1,\ldots,i_d} u = \partial_{i_1}\cdots\partial_{i_d} u$ parametrized by $i\in\{1,\ldots,d\}^m$. In Problem III, $T_{ij} u = \partial_i u_j$ for $i,j\in\{1,\ldots,d\}$. In Problem II, $Tu$ is as in Problem III, but there is one more entry, given by $\sqrt{c_0} \nabla\cdot u$.
	
	To bound the first term on the right side of \eqref{eq:regproofbmo} we set $v:=u-w$ and use, instead of \eqref{eq:decayharmonic}, the bound
	\begin{equation*}
		\int_{B_\rho(a)} |Tv - (Tv)_{a,\rho}|^2 \,dx \leq \underline C'' \left( \frac{\rho}{r} \right)^{d+2} \int_{B_r(a)} |Tv-(Tv)_{a,r}|^2 \,dx
		\qquad\text{for all}\ \rho\in(0,r] \,,
	\end{equation*}
	which is proved in essentially the same way as \eqref{eq:decayharmonic}. We bound the right side by
	\begin{align*}
		& 2 \underline C'' \left( \frac{\rho}{r} \right)^{d+2} \int_{B_r(a)} |Tu-(Tu)_{a,r}|^2 \,dx
		+ 2 \underline C'' \left( \frac{\rho}{r} \right)^{d+2} \int_{B_r(a)} |Tw-(Tw)_{a,r}|^2 \,dx \\
		\leq \, & 2 \underline C'' \left( \frac{\rho}{r} \right)^{d+2} \int_{B_r(a)} |Tu-(Tu)_{a,r}|^2 \,dx
		+ 2 \underline C'' \int_{B_r(a)} \tau(w) \,dx \,.
	\end{align*}
	The last term here is of the same form as the second term on the right side of \eqref{eq:regproofbmo}.
	
	For the second term on the right side of \eqref{eq:regproofbmo} we use the same bound \eqref{eq:regularityproof2} as before. 
	
	We see that $\Phi(\rho):= \int_{B_\rho(a)} |Tu - (Tu)_{a,\rho}|^2\,dx$ satisfies
	$$
	\Phi(\rho) \leq 4 \underline C'' \left( \frac{\rho}{r} \right)^{d+2} \Phi(r) + 4(1+\underline C'') C' r^d + 4(1+\underline C'')\Lambda(B_1)^{-1} r^{2m+\lambda} \|f\|_{L^{2,\lambda}}^2 \,.
	$$
	In view of the assumption $\lambda+2m>d$, we can bound in the last term $r^{2m+\lambda} \leq R^{2m+\lambda -d} r^d$. The assertion of the proposition now follows from the iteration lemma \cite[Lemma 5.13]{GiMa} as in the proof of Proposition \ref{regularity}.
\end{proof}


\section{Proof of the main result}\label{sec:proofmain}

Our goal in this section is to deduce Theorem \ref{main} from Propositions \ref{quasimin} and \ref{regularity}. We begin with a simple fact about functions whose derivatives belong to a Morrey space. We emphasize that $L^{2,d}(\R^d)=L^\infty(\R^d)$.

\begin{proposition}\label{campa}
	Let $m\in\N_0$ and $0\leq\mu\leq d$ and assume that $D^m u \in L^{2,\mu}(\R^d)$. If $\mu+2m>d$ we also assume that $u\in L^{2,0}(\R^d)$. Then
	$$
		u\in L^{2,\lambda}(\R^d)
		\qquad
		\begin{cases} 
			\text{with}\ \lambda = \mu+2m & \text{if}\ \mu+2m<d \,, \\
			\text{with any}\ \lambda<d & \text{if}\ \mu+2m=d \,, \\
			\text{with}\ \lambda = d & \text{if}\ \mu+2m> d \,.
		\end{cases}
	$$
\end{proposition}

\begin{proof}
	\emph{Step 1.} We begin by proving the assertion in the case $\mu+2m>d$.
	
	After decreasing $\mu$ if necessary, we may assume that $\mu+2m<d+2$. We claim that for every $k\in\N_0$ with $\mu+2k<d$, we have
	\begin{equation}
		\label{eq:morreyd}
			\int_{B_r(a)}  \left| D^{m-k} u(x) \right|^2\,dx \lesssim r^{\mu+2k} \\
			\qquad \text{for all}\ 0<r\leq r_0 \,,\ a\in\R^d \,.
	\end{equation}
	This can be shown by induction. Indeed, for $k=0$ this holds by assumption, and if $k\geq 1$ satisfies $\mu+2k<d$ and $D^{m-k+1}u\in L^{2,\mu+2(k-1)}$, then by Poincar\'e's inequality
	\begin{equation}
		\label{eq:poincare}
			\int_{B_r(a)}  \left| D^{m-k} u(x) - (D^{m-k})_{a,r} \right|^2\,dx \lesssim r^{\mu+2k} \qquad \text{for all}\ 0<r\leq r_0 \,,\ a\in\R^d \,.
	\end{equation}
	Since $\mu+2k<d$, this bound implies by (a variant of) \cite[Proposition~5.4]{GiMa} that \eqref{eq:morreyd} holds, thus completing the inductive proof of \eqref{eq:morreyd}.
	
	Since we are assuming $\mu+2m<d+2$, we can choose $k=m-1$ in \eqref{eq:morreyd}. Applying once again Poincar\'e's inequality, we obtain \eqref{eq:poincare} with $k=m$. Since $\mu+2m>d$, Campanato's theorem (see, e.g., \cite[Theorem 5.5]{GiMa}) implies that $u\in C^{0,(\mu+2m-d)/2}$. More precisely, we infer that
	$$
	|u(x)-u(y)| \lesssim |x-y|^{(\mu+2m-d)/2}
	\qquad\text{for all}\ x,y\in\R^d \ \text{with}\ |x-y|\leq r_0 \,.
	$$
	Using this and the assumption $u\in L^{2,0}(\R^d)$ we obtain
	\begin{align*}
		|u(x)|^2 \leq 2 |B_{r_0}(x)|^{-1} \int_{B_{r_0}(x)} |u(y)|^2 \,dy + 2  |B_{r_0}(x)|^{-1} \int_{B_{r_0}(x)} |u(x)-u(y)|^2 \,dy \lesssim 1 \,,
	\end{align*}
	that is, $u\in L^\infty = L^{2,d}$, as claimed.
	
	\medskip
	
	\emph{Step 2.} We argue by induction on $m$, the case $m=0$ being trivial. Thus let $m\geq 1$ and assume that the assertion for $m-1$ has already been proved for any $0\leq\mu\leq d$. The assertion for $\mu+2m>d$ follows from Step 1. Moreover, once we have shown the assertion in case $\mu+2m<d$, we obtain the assertion for $\mu+2m=d$ by noting that $D^m u\in L^{2,\tilde\mu}$ for any $\tilde\mu<\mu$.
	
	Thus, it remains to deal with the case $\mu+2m<d$. We apply Poincar\'e's inequality and arrive at \eqref{eq:poincare} with $k=1$. Since $\mu+2\leq\mu+2m<d$, we can again apply (a variant of) \cite[Proposition~5.4]{GiMa} and infer that \eqref{eq:morreyd} holds with $k=1$, that is, $D^{m-1}u\in L^{2,\mu+2}$. The assertion now follows by the induction hypothesis (with $\mu$ replaced by $\mu+2$).
\end{proof}

\begin{proof}[Proof of Theorem \ref{main}]
	Let $\Omega\subset\R^d$ be a quasi-open volume constrained minimizer of $\Lambda$ and let $u$ be a corresponding $L^2$-normalized eigenfunction, extended by zero to $\R^d$. If we can show that $u\in C^{m-1,\alpha}(\R^d)$ for any $\alpha<1$, then 
	$$
	\tilde\Omega:=\bigcup_{k=0}^{m-1} \{D^k u\neq 0\}
	$$ 
	is an open set that satisfies $\Lambda(\tilde\Omega)\leq\Lambda(\Omega)$ (by taking $u\in H^m_0(\tilde\Omega)$ as a trial function, see \cite[Lemma 10]{Le}) and $|\tilde\Omega|\leq|\Omega|$ (since $\tilde\Omega\subset\Omega$ up to sets of capacity zero), so by minimality both inequalities need to be equalities, and $\Omega$ is open up to a set of measure zero, as claimed.
	
	\medskip
	
	Therefore, it suffices to prove regularity. Our goal is to prove
	\begin{equation}
		\label{eq:goal}
			D^m u \in L^{2,\mu}
			\qquad\text{for all}\ \mu<d \,.
	\end{equation}
	Indeed, once this is shown, we can argue similarly as in the proof of Proposition \ref{campa}. First, by Poincar\'e's	inequality, we obtain \eqref{eq:poincare} with $k=1$ and then, by Campanato's theorem \cite[Theorem 5.5]{GiMa}, $D^{m-1} u\in C^{0,(\mu+2-d)/2}$ when $\mu+2>d$. Since $\mu$ in \eqref{eq:goal} can be taken arbitrarily close to $d$, we obtain the H\"older regularity assertion in the theorem. According to Proposition \ref{campa}, the boundedness assertion for $u$ follows from \eqref{eq:goal} with $\mu+2m>d$. That for $D^ku$ with $1\leq k\leq m-1$ follows similarly. 
	
	It remains to prove \eqref{eq:goal}. We start with the information $u\in L^2(\R^d)\subset L^{2,0}(\R^d)$. By Proposition \ref{quasimin}, $u$ is a local quasiminimizer of the functional $J_f$ with $f\in L^{2,0}$, so by Proposition \ref{regularity} $D^m u\in L^{2,\mu}$ with $\mu=2m$ if $2m<d$ and with any $\mu<d$ if $2m\geq d$. Thus, when $2m\geq d$ we have reached our goal \eqref{eq:goal}. 
	
	In what follows we assume that $2m<d$ and we apply a bootstrap. According to Proposition \ref{campa}, the fact that $D^m u\in L^{2,2m}$ implies that $u\in L^{2,\lambda}$ with $\lambda = 4m$ if $4m<d$, with any $\lambda<d$ if $4m=d$ and with $\lambda=d$ if $4m>d$. In each case we have improved our initial information $u\in L^{2,0}$. Thus, we can iterate the argument, using again that $u$ is a local quasiminimizer of $J_f$, where now $f\in L^{2,\lambda}$. We infer that $D^m u\in L^{2,\mu}$ with $\mu=\lambda+2m$ if $\lambda+2m<d$ and with any $\mu<d$ if $\lambda+2m\geq d$. Recalling the values of $\lambda$, we obtain $\mu=6m$ if $6m<d$ and any $\mu<d$ if $6m\geq d$. Thus, when $6m\geq d$ we have reached our goal \eqref{eq:goal}.
	
	Assuming $6m<d$, we have improved the knowledge $D^m u\in L^{2,2m}$ to $D^m u\in L^{2,6m}$ and using Proposition \ref{campa} we can translate this into an improvement of the Morrey space membership of $u$. Continuing in this way we arrive after finitely many iterations at \eqref{eq:goal}. This completes the proof of Theorem \ref{main}.
\end{proof}

\begin{remark}\label{rem:bmo}
	The proof shows that after finitely many steps we have $u\in L^{2,\lambda+2m}(\R^d)$ and therefore, by Propositions \ref{quasimin} and \ref{regularitybmo}, $D^mu$ satisfies the BMO-type condition \eqref{eq:regularitybmo}. This seems to be the best regularity that can be obtained with our method.
\end{remark}


\subsection*{Acknowledgements}
The author is grateful to R.~Leylekian, A.~Henrot, A.~Lemenant and Y.~Privat for helpful remarks and suggestions.

\subsection*{Funding}
Partial support through the German Research Foundation grants EXC-2111-390814868 and TRR 352-Project-ID 470903074, as well as through the US National Science Foundation grant DMS-1954995 is acknowledged.


\bibliographystyle{amsalpha}

\end{document}